\documentclass{amsart}

\usepackage{amsfonts, amsmath, amscd}
\usepackage[psamsfonts]{amssymb}
\usepackage{amssymb}
\usepackage{hyperref}

\usepackage{placeins}
\usepackage{ragged2e}

\usepackage[all]{xy}

\numberwithin{equation}{section}

\newtheorem{theorem}{Theorem}[section]
\newtheorem{corollary}[theorem]{Corollary}

\theoremstyle{definition}

\theoremstyle{plain}

\input xy
\xyoption{all}

\newcommand{\NN}{\mathbb{N}}

\newcommand{\RR}{\mathbb{R}}

\renewcommand{\phi}{\varphi}

\begin{document}

\center{\bf\LARGE Tweaking Ramanujan's Approximation of $n!$}
\bigskip

\center{ \bf\Large by}

\bigskip

\center{\bf\Large SIDNEY A. MORRIS}

\bigskip

\center{School of Engineering, Information Technology and Physical Sciences,} \center{Federation University Australia,} 
\center{PO Box 663, Ballarat, Victoria, 3353,  Australia \& }
\center{Department of Mathematics and Statistics,}
\center{  La~Trobe University, Melbourne, Victoria, 3086, Australia}
\center{\LARGE morris.sidney@ gmail.com}
\center{\Large ORCID: 0000-0002-0361-576X}

\justify

\bigskip \bigskip

{\bf Funding:} The author received no funding for this research.

\noindent {\bf Conflicts of interest:}  There are no conflicts of interest.

\medskip

\noindent\Large {\bf Keywords:}  n!, gamma function, approximation, asymptotic, Stirling formula, Ramanujan.

\smallskip

\noindent\Large {\bf 2010 Mathematics Subject Classification:} Primary 33B15; Secondary 41A25

\bigskip\bigskip
\noindent{\bf\Large Abstract:}\Large\ About 1730 James Stirling, building on the work of Abraham de Moivre, published what is known as Stirling's approximation of $n!$. He gave a good formula which is asymptotic to $n!$. Since  then hundreds of papers have given alternative proofs of his result and improved upon it, including notably by Burside, Gosper,  and Mortici. However Srinivasa Ramanujan gave a remarkably better asymptotic formula. Hirschhorn and Villarino gave a nice proof of Ramanujan's result and an error estimate for the approximation. In recent years there have been several  improvements of Stirling's formula including  by Nemes,  Windschitl, and Chen.  Here it is shown (i)  how all these asymptotic results can be easily verified; (ii) how Hirschhorn and Villarino's argument allows a tweaking of Ramanujan's result to give a better approximation; (iii) that a new asymptotic formula can be obtained by further tweaking of Ramanujan's result;  (iv) that Chen's asymptotic formula is better than the others mentioned here, and the new asymptotic formula is comparable with Chen's.

\newpage

\title{Tweaking Ramanujan's Approximation of n!}

\author{Sidney A. Morris}

\address{School of Engineering, Information Technology and Physical Sciences, Federation University Australia, PO Box 663, Ballarat, Victoria, 3353,  Australia \& Department of Mathematics and Statistics,  La~Trobe University, Melbourne, Victoria, 3086, Australia}
\email{morris.sidney@gmail.com}
\keywords{n!, gamma function, approximation, asymptotic, Stirling formula, Ramanujan}
\subjclass[2010]{Primary 33B15; Secondary 41A25}

\begin{abstract}
In 1730 James Stirling, building on the work of Abraham de Moivre, published what is known as Stirling's approximation of $n!$. He gave a good formula which is asymptotic to $n!$. Since  then hundreds of papers have given alternative proofs of his result and improved upon it, including notably by Burside, Gosper,  and Mortici. However Srinivasa Ramanujan gave a remarkably better asymptotic formula. Hirschhorn and Villarino gave a nice proof of Ramanujan's result and an error estimate for the approximation. In recent years there have been several  improvements of Stirling's formula including  by Nemes,  Windschitl, and Chen.  Here it is shown (i)  how all these asymptotic results can be easily verified; (ii) how Hirschhorn and Villarino's argument allows a tweaking of Ramanujan's result to give a better approximation; (iii) that a new asymptotic formula can be obtained by further tweaking of Ramanujan's result;  (iv) that Chen's asymptotic formula is better than the others mentioned here, and the new asymptotic formula is comparable with Chen's.
\end{abstract}

\maketitle

\section{Introduction}
About 1730 James Stirling, building on the work of Abraham de Moivre, published what is known as Stirling's approximation of $n!$.  In fact,  Stirling \cite{Tweddle} proved that $n!\sim\sqrt{2\pi n} \left(\dfrac{n}{e}\right)^n$; that is, $n!$ is asymptotic to  $\sqrt{2\pi n} \left(\dfrac{n}{e}\right)^n$. De Moivre had been  considering a gambling problem and needed to approximate $\binom{2n}{n}$ for large $n$. The Stirling approximation gave a very satisfactory solution to this problem.

The problem of extending the factorial from the positive integers to a wider class of numbers  was first investigated by Daniell Bernoulli and Christian Goldbach in the 1720s. In  1729 Leonhard Euler  succeeded and  in 1730 he proved that  for $z$ any complex number with positive real part,  $\Gamma(z)= \int\limits_0^\infty t^{z-1}e^t\,dt$, where  $\Gamma(n)=(n-1)!$, for any positive integer $n$. The name gamma function is due to Adrien-Marie Legendre.

In 1774 Pierre-Simon Laplace noticed that Stirling's formula for $n!$ has a generalization to the gamma function, namely that for $x$ a positive real number, $\Gamma(x+1)\sim \sqrt{2\pi x} \left(\dfrac{x}{e}\right)^x $. One of the most elementary proofs of Stirling's formula for the gamma function  is by Reinhard Michel \cite{Michel}. 

Most of the proofs in the literature of Stirling's formula and its extensions prove  that they are asymptotic by establishing an error estimate such as  
$$\Gamma(x+1)=\sqrt{2\pi x}\left(\frac{x}{e}\right)^x(1+O\left(x^{-1})\right). $$
In fact most of the effort goes into proving such error estimates.

In this paper we observe that once one knows that Stirling's formula is asymptotic to   $\Gamma(x+1)$, all of the other known asymptotic formulae can be verified trivially without the need to establish any error estimates. 

In 1917 William Burnside \cite{Burnside} published a modest improvement on Stirling's formula, namely  $\Gamma(x+1)\sim \sqrt{2\pi}\left(\dfrac{x+1/2}{e}\right)^{x+1/2}$. How modest an improvement it  is can be ascertained from  Table 1 below. In 1978 Ralph William (Bill) Gosper Jr, \cite{Gosper}, published a significant improvement on  Stirling and Burnside's formulae. It was that  $\Gamma(x+1)\sim \sqrt{\pi}\left(\dfrac{x}{e}\right)^x\sqrt{2x+\dfrac{1}{3}}$. 
In a web post in 2002, Robert H. Windschitl, \cite{Smith}, gave an elegant and good  asymptotic approximation of $n!$, namely  that $\Gamma(x+1)\sim \sqrt{2\pi x} \left(\dfrac{x}{e}\right)^x \left(x \sinh\left(\dfrac{1}{x}\right)\right)^{\frac{x}{2}}$. In  2010 Gerg\H{o} Nemes gave an asymptotic approximation which is almost as good as Windschitl's but better than all the others at that time. It was that $\Gamma(x+1)\sim \sqrt{2\pi x} \left(\dfrac{x}{e}\right)^x\left(1+\dfrac{1}{12x^2-\frac{1}{10}}\right)^x$. An asymptotic formula of a different style, which is much better than Gosper's, was published in 2011 by Cristinel Mortici \cite{Mortici}.  It was $\Gamma(x+1)\sim \sqrt{2\pi\,x}\left(\dfrac{x}{e}+\dfrac{1}{12\,e\,x}\right)^x$.

Pierre-Simon Laplace discovered what is now known as the Stirling series for the gamma function.  \vspace{-15pt}

\begin{align*}\Gamma(x+1)\sim e^{-x}x^{x+\frac{1}{2}}\sqrt{2\pi}\bigg(1 +& \frac{1}{12x}+\frac{1}{288x^2}-\frac{139}{51840x^3} -\frac{571}{2488320x^4}\\+
&\sum_{n=5}^\infty \frac{a_n}{b_nx^{n}}\bigg),\end{align*}
where the  real numbers $a_n$ and $b_n$ are explicitly calculated  in \cite{Nemes1}.
As stated in \cite{Namias}, ``the performance deteriorates as the number of terms is increased beyond a certain value''.  In Table 2 below we show how using up to the term $x^{-4}$ in this divergent  series compares with the other approximations.

A major advance in producing an asymptotic formula for $n!$ was made  by the extraordinary  Indian mathematician Srinivasa Ramanujan   (1887--1920) in the last year of his life. Ramanujan's claim, recorded in \cite[p.\,339]{Ramanujan}, was  that  
$$\Gamma(x+1)=\sqrt{\pi}\left(\frac{x}{e}\right)^x\left(8x^3+4x^2+x+\frac{\theta_x}{30}\right)^{\frac{1}{6}},$$
where $\theta_x\to 1$ as $x\to \infty$ and $\dfrac{3}{10}<\theta_x<1$ and he gave numerical evidence for  his claim.  

Ramanujan's approximation  is substantially better than all those which were published in the subsequent 80 years.  For example,  when $n=$ 1 million, the percentage error of Ramanujan's approximation is one million million times better than Gosper's.  

In 2013 Michael Hirschhorn and  Mark B. Villarino \cite{Hirschhorn} proved the correctness of Ramanujan's claim above  for positive integers. They showed that Ramanujan's $\theta_n$ satisfies for each positive integer $n$:
$$1-\frac{11}{8n}+\frac{79}{112n^2}<\theta_n<1-\frac{11}{8n}+\frac{79}{112n^2} +\frac{20}{33n^3}.$$

\noindent Although they did not explicitly say it, it is clear from their work that 
 $\Gamma(x+1)\sim \sqrt{\pi}\left(\dfrac{x}{e}\right)^x \left(8x^3+4x^2+x+\dfrac{ 1-\frac{11}{8x}+\frac{79}{112x^2}}{30}\right)^{\frac{1}{6}}$, at least for positive integers. This  approximation, as can be seen in Table 3, is better than all that preceded it. Indeed for $n=1$ million, it has a percentage error at least one million times better than each one.  

In 2016 Chao-Ping Chen \cite{Chen} produced an asymptotic approximation which for $n=1$ million has a percentage error one million times better than that of Hirschhorn and Villarino. His asymptotic approximation is 
$$\Gamma(x+1)\sim \sqrt{2\pi x} \left(\dfrac{x}{e}\right)^x \left(1+ \dfrac{1}{12x^3+\frac{24}{7}x-\frac{1}{2}}\right)^{x^2+ \frac{53}{210}}.$$

A more detailed analysis of   Hirschhorn and Villarino's improvement on that of Ramanujan, suggests a tweaking of their approximation. That tweaking produces an approximation which is stated in Corollary \ref{2.3} and is comparable to Chen's for $n=1$ to $n=10,000$ and much better than Chen's for $n=1$ million, as is evidenced in Table 3. 

Let me make, with some hesitation, a \emph{controversial} remark. Chen points out that 
Burnside's approximation involves an error of order $O(n^{-1})$, Ramunajan's approximation involves an error of $O(n^{-4})$, Nemes and Windschitl's approximations involves an error of $O(n^{-5}$), and his own approximation involves an error of order $O(n^{-7})$. But in my opinion, these statements are not very informative not only because all the approximations are asymptotic to $n!$, but also because of the following extreme example: 
$$\sqrt{2\pi x} \left(\dfrac{x}{e}\right)^x \left(1+ \dfrac{1}{12x^3+\frac{24}{7}x-\frac{1}{2}}\right)^{x^2+ \frac{53}{210}}\bigg(1 +\frac{10^{100}}{n^8}\bigg)\sim n!$$
and has an error of the order of $O(n^{-7})$ but is an absurdly bad approximation even for $n=1$ million. The order estimate can be used to compare approximations for ``very large'' $n$, but does not tell us how large is ``very large''.

\bigskip

\section{The Approximations of $\Gamma(x+1)$}

As suggested in $\S$1, once we know Stirling's asymptotic formula for $\Gamma(x+1)$, all of the others follow trivially. This fact is captured in Theorem \ref{2.1} .

\begin{theorem}\label{2.1} Let $f$ be a function from a subset $(a,\infty)$ to $\RR$, where $a\in \RR, a>0$. If $\lim\limits_{x\to \infty} f(x)=1$, then $\Gamma(x+1)\sim \sqrt{2\pi \,x}\left(\frac{x}{e}\right)^x.f(x)$. \end{theorem}

\begin{proof} This follows immediately from the Stirling asymptotic approximation, namely that  $\Gamma(x+1)\sim \sqrt{2\pi \,x}\left(\frac{x}{e}\right)^x$.  \end{proof}

As an immediate corollary of Theorem \ref{2.1} we obtain that all of the other mentioned approximations are asymptotic to $\Gamma(x+1)$. Some of these were proved by the authors only for $x$ a positive integer.

\begin{corollary}\label{2.2} For  $x$ a positive real number:
\begin{itemize}
\item[(i)] {\rm Burnside \cite{Burnside}}:\quad
 $\Gamma(x+1)\sim \sqrt{2\pi}\left(\dfrac{x+1/2}{e}\right)^{x+1/2}$\,; 
 \item[(ii)] {\rm Gosper \cite{Gosper}}: $\Gamma(x+1)\sim \sqrt{\pi}\left(\dfrac{x}{e}\right)^x\sqrt{2x+\dfrac{1}{3}}$\,;
 \item[(iii)] {\rm Mortici \cite{Mortici}}:\quad $\Gamma(x+1)\sim \sqrt{2\pi\,x}\left(\dfrac{x}{e}+\dfrac{1}{12\,e\,x}\right)^x$\,; 
 \item[(iv)] {\rm Ramanujan \cite{Ramanujan}}: $\Gamma(x+1)\sim \sqrt{\pi}\left(\dfrac{x}{e}\right)^x\left(8x^3+4x^2+x+\dfrac{1}{30}\right)^{\frac{1}{6}}$\,;
 \item[(v)] {\rm Laplace $(n)$}:\quad Fix $n\in \NN$. For $a_i,b_i\in \NN$,  $$\quad\quad \Gamma(x+1)\sim e^{-x}x^{x+\frac{1}{2}}\sqrt{2\pi}\bigg(1 + \frac{1}{12x}+\frac{1}{288x^2}+
\sum_{i=3}^n \frac{a_i}{b_i x^{i}}\bigg);$$
 \item[(vi)]  {\rm Nemes:}\quad $\Gamma(x+1)\sim \sqrt{2\pi x} \left(\dfrac{x}{e}\right)^x\left(1+\dfrac{1}{12x^2-\frac{1}{10}}\right)^x$.
 \item[(vii)] {\rm Windschitl \cite{Smith}:} $\Gamma(x+1)\sim \sqrt{2\pi x} \left(\frac{x}{e}\right)^x \left(x \sinh\left(\frac{1}{x}\right)\right)^{\frac{x}{2}}$.
 \item[(viii)] {\rm Hirschhorn \& Villarino \cite{Hirschhorn} }:\newline $\Gamma(x+1)\sim \sqrt{\pi}\left(\dfrac{x}{e}\right)^x \left(8x^3+4x^2+x+\dfrac{ 1-\frac{11}{8x}+\frac{79}{112x^2}}{30}\right)^{\frac{1}{6}}$\,.
 \item[(ix)] {\rm Chen \cite{Chen}:} $\Gamma(x+1)\sim \sqrt{2\pi x} \left(\dfrac{x}{e}\right)^x \left(1+ \dfrac{1}{12x^3+\frac{24}{7}x-\frac{1}{2}}\right)^{x^2+ \frac{53}{210}}$.
\end{itemize}
\end{corollary} 
\begin{proof} In each case it is sufficient to determine the function $f$ in Theorem \ref{2.1}and observe that $\lim\limits_{x\to \infty}f(x)=1$. 
\begin{itemize}
\item[(i)] Use $f(x)=\left(1+\dfrac{1}{2x}\right)^x\left(\dfrac{1+\dfrac{1}{2x}}{e}\right)^{\frac{1}{2}}$.
\item[(ii)] Use $f(x)= \sqrt{1+\dfrac{1}{6x}}\,  $.
\item[(iii)] Use $f(x)=\left(1+\dfrac{1}{12x^2}\right)^x\,.$
\item[(iv)] Use $f(x)={\left(1+\dfrac{1}{2x}+\dfrac{1}{8x^2}  +\dfrac{1}{240x^3}\right)^{\frac{1}{6}}}\,.$
\item[(v)] Use $f(x)= \left(1+\dfrac{1}{12x}+\dfrac{1}{288x^2}+\sum\limits_{i=3}^n\dfrac{a_i}{b_ix^i}\right)$.\vspace{4pt}
\item[(vi)] Use $f(x)= \left(1+\dfrac{1}{12x^2-\frac{1}{10}}\right)^x$. \vspace{4pt}
\item[(vii)] Use $ f(x)=\left(x \sinh\left(\frac{1}{x}    \right)\right)^{\frac{x}{2}}$.
\item[(viii)]Use $f(x)={\left(1+\dfrac{1}{2x}+\dfrac{1}{8x^2}  +
\dfrac{1-\frac{11}{8x}+\frac{79}{112x^2}}{240x^3}
\right)^{\frac{1}{6}}}\,.$ \vspace{4pt}
\item[(ix)] Use $f(x)=\left(1+ \dfrac{1}{12x^3+\frac{24}{7}x-\frac{1}{2}}\right)^{x^2+ \frac{53}{210}}$.
\end{itemize}
\end{proof}

In fact, at the expense of a little more complication, we can tweak Ramanujan's approximation again to get an even better approximation for large values of $x$, which we refer to in the table below as the SAM approximation. The proof of the corollary uses an obvious modification of the proof of (ix) above. 

\begin{corollary}\label{2.3} For $x$ a positive real number, \newline
$$\Gamma(x+1)\sim \sqrt{\pi}\left(\dfrac{x}{e}\right)^x \left(8x^3+4x^2+x+\dfrac{ 1-\dfrac{11}{8x}+\dfrac{79}{112x^2} +\dfrac{A}{x^3}  }{30}\right)^{\frac{1}{6}},$$ where $A=\dfrac{380279456577}{722091376690}$. \qed
\end{corollary}

\section{Numerical Analysis of the Approximations}

The tables in this section were calculated using the WolframAlpha software package. (See \url{https://www.wolframalpha.com/}.) They demonstrate the performance of the asymptotic approximations.

\smallskip
Each of the approximations gets further and further from $n!$ as $n$ tends to infinity. So the quality of the approximations is best judged by considering the percentage error, that is $100\times \dfrac{\text{approximation}-n!}{n!}$. 

In the tables S $=$ Stirling, B $=$ Burnside, G $=$ Gosper, L4 $=$ (Laplace) Stirling series up to $x^{-4}$, M $=$ Mortici, R $=$ Ramanujan, HV $=$ Hirschhorn and Villarino, C  $=$ Chen, and SAM $=$ the author of this paper.

\smallskip
From the  tables it is abundantly clear that Gosper's approximation is a much better approximation than Stirling's, and Mortici's elegant approximation is closer in accuracy to Ramanujan's. Ramanujan's approximation is amazingly good. The tweaking of Ramanujan's approximation using the Hirschhorn-Villarino results significantly improves the approximation. Chen's approximation is better than all that precede it. The SAM approximation obtained by extra tweaking of Ramanujan's approximation produces  an approximation similar to Chen's up to $n=10,000$ and much better for $n=1,000,000$. 

 \begin{table}[h!]   \caption{}
\centering
 \begin{tabular}{|c|c|c|c|c|} 
 \hline 
$n$&$n!$ & S \%error & B \%error&G \%error\\ [0.5ex] 
 \hline
 2 & 2&4.0&1.7&1.3$\times10^{-1}$\\[0.5ex] 
 \hline
5&1.2$\times10^{2}$&1.7&7.6$\times10^{-1}$&2.5$\times10^{-2}$\\[0.5ex] 
 \hline
 10&3.6$\times10^{6}$&8.3$\times10^{-1}$&4.0$\times10^{-1}$ &6.6$\times10^{-3}$\\
  \hline  
  20&   2.4$\times 10^{18}$&4.2$\times10^{-1}$& 2.0$\times10^{-1}$&1.7$\times10^{-3}$ \\[0.5ex] 
  \hline
  50 &   3.0$\times 10^{64}$& 1.7$\times10^{-1}$&8.3$\times10^{-2}$&2.7$\times10^{-4}$\\[0.5ex] 
  \hline
  100 &9.3$\times 10^{157}$  &8.3$\times10^{-1}$&4.1$\times10^{-2}$ &6.9$\times10^{-5}$\\[0.5ex] 
  \hline
   $10^3$ &4.0$\times10^{2567}$&8.3$\times10^{-3}$&4.2$\times10^{-3}$  &6.9$\times10^{-7}$\\ [0.5ex] \hline
   $10^{4}$&2.8$\times10^{35659}$&8.3$\times10^{-4}$&4.2$\times10^{-4}$&6.9$\times10^{-9}$\\ [0.5ex] \hline
   $10^{6}$&8.3$\times10^{5565708}$&8.3$\times10^{-6}$ &  4.2$\times10^{-6}$ &6.9$\times10^{-13}$\\ [0.5ex] \hline
\end{tabular}
\end{table}

\bigskip

 \begin{table}[h!] 
\centering \caption{}
 \begin{tabular}{|c|c|c|c|c|c|} 
 \hline 
$n$&$n!$         & M \%error           &R \% error                &L4 \%error & N \%error\\ [0.5ex] 
 \hline
 2 & 2             &1.0$\times10^{-2}$  &3.3$\times10^{-3}$   &1.4$\times10^{-2}$   &1.7$\times10^{-3}$    \\[0.5ex] 
 \hline
5&1.2$\times10^{2}$&5.7$\times10^{-4}$  &1.2$\times10^{-4}$&3.5$\times10^{-4}$     &2.0$\times10^{-5}$   \\[0.5ex] 
 \hline
 10&3.6$\times10^{6}$ &7.0$\times 10^{-5}$&8.6$\times10^{-6}$&7.8$\times10^{-7}$      &6.5$\times10^{-7}$  \\[0.5ex] 
  \hline  
  20&   2.4$\times 10^{18}$ &8.7$\times10^{-6}$& 5.7$\times10^{-7}$&2.4$\times10^{-8}$        &2.0$\times10^{-8}$  \\[0.5ex] 
  \hline
  50 &   3.0$\times 10^{64}$ & 5.6$\times10^{-7}$&1.5$\times10^{-8}$ &2.5$\times10^{-10}$     &2.1$\times10^{-10}$  \\[0.5ex] 
  \hline
  100 &9.3$\times 10^{157}$  &6.9$\times10^{-8}$& $9.5\times 10^{-10}$&7.8$\times10^{-12}$       &6.5$\times10^{12}$              \\[0.5ex] 
  \hline
   $10^3$&4.0$\times10^{2567}$&6.9$\times10^{-11}$&9.5$\times10^{-14}$&7.8$\times10^{-17}$     &6.5$\times10^{-17}$ \\ [0.5ex] \hline
  $ 10^{4}$&2.8$\times10^{35659}$&6.9$\times10^{-14}$&9.5$\times10^{-18}$&    7.8$\times10^{-22}$ &6.5$\times10^{-22}$  \\ [0.5ex] \hline
   $10^{6}$&8.3$\times10^{5565708}$&6.9$\times10^{-20}$&9.5$\times10^{-26}$&7.8$\times10^{-32}$       & 6.5$\times10^{-32}$\\ [0.5ex] \hline
\end{tabular}
\end{table}


\bigskip
 \begin{table}[h!] \caption{}
\centering
 \begin{tabular}{|c|c|c|c|c|c|} 
 \hline 
$n$&$n!$ &W \% error        & HV \%error           &C \% error                &SAM \%error\\ [0.5ex] 
 \hline
 2 & 2 &1.6$\times10^{-3}$               &1.6$\times10^{-4}$&2.2$\times 10^{-4}$&2.9$\times10^{-4}$\\[0.5ex] 
 \hline
5&1.2$\times10^{2}$ &1.9$\times10^{-5}$      &1.5$\times10^{-6}$&5.0$\times10^{-7}$&6.0$\times10^{-7}$
\\[0.5ex] 
 \hline
 10&3.6$\times10^{6}$ &6.1$\times10^{-7}$     &3.0$\times10^{-8}$&4.1$\times10^{-9}$&4.9$\times10^{-9}$
 \\[0.5ex] 
  \hline  
  20&   2.4$\times 10^{18}$ &1.9$\times10^{-8}$       &5.2$\times10^{-10}$&3.2$\times10^{-11}$&3.8$\times10^{-11}$  \\[0.5ex] 
  \hline
  50 &   3.0$\times 10^{64}$  &2.1$\times10^{-10}$        &2.3$\times10^{-12}$&5.3$\times10^{-14}$&6.3$\times10^{-14}$\\[0.5ex] 
  \hline
  100 &9.3$\times 10^{157}$  & 6.2$\times10^{-12}$           &3.6$\times10^{-14}$& 4.2$\times10^{-16}$ &4.9$\times10^{-16}$           \\[0.5ex] 
  \hline
   $10^3$&4.0$\times10^{2567}$&6.2$\times10^{-17}$      &3.7$\times10^{-20}$&4.17$\times10^{-23}$&4.9$\times10^{-23}$\\ [0.5ex] \hline
  $ 10^{4}$&2.8$\times10^{35659}$&6.2$\times10^{-22}$       &3.7$\times10^{-26}$&4.2$\times10^{-30}$&4.9$\times10^{-30}$\\ [0.5ex] \hline
   $10^{6}$&8.3$\times10^{5565708}$& 6.2$\times10^{-32}$       &3.7$\times10^{-38}$&4.2$\times10^{-44}$&1.3$\times10^{-50}$\\ [0.5ex] \hline
\end{tabular}
\end{table}
\FloatBarrier


\end{document}